\documentclass[12pt]{article}
\usepackage[latin1]{inputenc}
\usepackage[british]{babel}
\usepackage{cmap}
\usepackage{lmodern}

\usepackage{amssymb, amsmath, amsthm}
\usepackage[a4paper,top=25mm,bottom=25mm,left=25mm,right=25mm]{geometry}
\usepackage{ragged2e}

\usepackage{authblk} 
\usepackage{pifont}
\usepackage{graphicx}
\usepackage[usenames,dvipsnames,svgnames,table]{xcolor}
\usepackage[figuresright]{rotating}
\usepackage{xtab} 
\usepackage{longtable} 
\usepackage{multirow}
\usepackage{footnote}
\usepackage[stable]{footmisc}
\usepackage{chngpage} 
\usepackage{pdflscape} 
\usepackage[nottoc,notlot,notlof]{tocbibind} 

\usepackage{pgfplots}
\pgfplotsset{every tick label/.append style={font=\footnotesize}}
\pgfplotsset{compat=1.14}
\usepackage{setspace}

\usepackage{array}
\newcolumntype{K}[1]{>{\centering\arraybackslash$}p{#1}<{$}}

\makesavenoteenv{tabular}
\usepackage{tabularx}
\usepackage{booktabs}
\usepackage{threeparttable}
\usepackage[referable]{threeparttablex} 
\newcolumntype{R}{>{\raggedleft\arraybackslash}X}
\newcolumntype{L}{>{\raggedright\arraybackslash}X}
\newcolumntype{C}{>{\centering\arraybackslash}X}
\newcolumntype{A}{>{\columncolor{gray!25}}C}
\newcolumntype{a}{>{\columncolor{gray!25}}c}

\newlength{\tablen}

\usepackage{dcolumn} 
\newcolumntype{.}{D{.}{.}{-1}}

\usepackage{tikz}
\usetikzlibrary{arrows, calc, matrix, patterns, positioning, trees}
\usepackage[semicolon]{natbib} 
\usepackage[hyphens]{url}
\usepackage{hyperref} 
\hypersetup{
  colorlinks   = true,    		
  urlcolor     = blue,    		
  linkcolor    = blue,    		
  citecolor    = ForestGreen	
}
\usepackage{microtype}
\usepackage[justification=centering]{caption} 

\usepackage[labelformat=simple]{subcaption}

\DeclareCaptionLabelFormat{parenthesis}{(#2)}
\makeatletter
\renewcommand\p@subfigure{\arabic{figure}.}
\makeatother

\DeclareCaptionLabelFormat{parenthesis}{(#2)}
\makeatletter
\renewcommand\p@subtable{\arabic{table}.}
\makeatother

\usepackage{enumitem}

\setlist[itemize]{leftmargin=2.5\parindent}
\setlist[enumerate]{leftmargin=2.5\parindent}

\def\addlegendimage{\csname pgfplots@addlegendimage\endcsname}

\theoremstyle{plain}

\newtheorem{lemma}{Lemma}

\newtheorem{theorem}{Theorem}

\theoremstyle{definition}

\newtheorem{definition}{Definition}[section]
\newtheorem{example}{Example}

\theoremstyle{remark}

\newtheorem{remark}{Remark}

\makeatletter
\let\@fnsymbol\@alph
\makeatother

\def\keywords{\vspace{.5em} 
{\noindent \textit{Keywords}: }}

\def\AMS{\vspace{.5em} 
{\noindent \textbf{\emph{MSC} class}: }}

\def\JEL{\vspace{.5em} 
{\noindent \textbf{\emph{JEL} classification number}: }}

\title{On the coincidence of optimal completions for small pairwise comparison matrices with missing entries}
\author{\href{https://sites.google.com/view/laszlocsato}{L\'aszl\'o Csat\'o}\thanks{~Corresponding author. Email: \emph{laszlo.csato@sztaki.hu} \newline
Institute for Computer Science and Control (SZTAKI), E\"otv\"os Lor\'and Research Network (ELKH), Laboratory on Engineering and Management Intelligence, Research Group of Operations Research and Decision Systems, Budapest, Hungary \newline
Corvinus University of Budapest (BCE), Institute of Operations and Decision Sciences, Department of Operations Research and Actuarial Sciences, Budapest, Hungary}
$\qquad \qquad$
Kolos Csaba \'Agoston\thanks{~Email: \emph{kolos.agoston@uni-corvinus.hu} \newline Corvinus University of Budapest (BCE), Institute of Operations and Decision Sciences, Department of Operations Research and Actuarial Sciences, Budapest, Hungary}
$\qquad \qquad$
S\'andor Boz\'oki\thanks{~Email: \emph{sandor.bozoki@sztaki.hu} \newline
Institute for Computer Science and Control (SZTAKI), E\"otv\"os Lor\'and Research Network (ELKH), Laboratory on Engineering and Management Intelligence, Research Group of Operations Research and Decision Systems, Budapest, Hungary \newline
Corvinus University of Budapest (BCE), Institute of Operations and Decision Sciences, Department of Operations Research and Actuarial Sciences, Budapest, Hungary}
}

\date{\today}

\def\Dedication{
\begin{small}
{\noindent
``\emph{We conclude that if the dimension is small, and the researcher knows a priori that the errors are small, and there is exactly one judgment for each pair of entities, there is little (only time, effort, and a little accuracy) to argue against using the eigenvector.}''\footnote{~Source: \citet[p.~405]{CrawfordWilliams1985}.}
}
\end{small}

\vspace{0.5cm} 
\justify }

\begin{document}

\maketitle
\thispagestyle{empty}
\Dedication

\begin{abstract}
\noindent
Incomplete pairwise comparison matrices contain some missing judgements. A natural approach to estimate these values is provided by minimising a reasonable measure of inconsistency after unknown entries are replaced by variables. Two widely used inconsistency indices for this purpose are Saaty's inconsistency index and the geometric inconsistency index, which are closely related to the eigenvector and the logarithmic least squares priority deriving methods, respectively. The two measures are proven to imply the same optimal filling for incomplete pairwise comparison matrices up to order four but not necessarily for order at least five.

\keywords{Analytic Hierarchy Process (AHP); decision analysis; eigenvalue method; incomplete pairwise comparisons; logarithmic least squares method}

\AMS{90-10, 90B50, 91B08}

\JEL{C44, D71}
\end{abstract}

\clearpage

\section{Introduction} \label{Sec1}

The pairwise comparison methodology is applied in many decision-making frameworks such as the famous Analytic Hierarchy Process (AHP), proposed by Saaty \citep{Saaty1977, Saaty1980}. Since the aim is to obtain the priorities for the alternatives, the resulting pairwise comparison matrix should be transformed into a weight vector. To that end, a number of procedures have been suggested \citep{ChooWedley2004}. The two most popular techniques are the eigenvector \citep{Saaty1977} and the logarithmic least squares (often called geometric mean) \citep{CrawfordWilliams1985} methods. They lead to the same solution if the pairwise comparisons are \emph{consistent}, that is, the direct comparison of alternatives $i$ and $j$ coincides with any indirect comparison of them through a third alternative $k$. In addition, the two algorithms are verified to give the same weights if the number of alternatives is at most three \citep[p.~393]{CrawfordWilliams1985}.

Several studies have examined the similarity of the weight vectors derived by the eigenvector and the logarithmic least squares methods for at least four alternatives. According to the Monte Carlo simulations of \citet{HermanKoczkodaj1996}, the priorities are generally closer if the matrix is less inconsistent. \citet{KulakowskiMazurekStrada2022} give analytical proof of the convergence. \citet{MazurekKulakowskiErnstStrada2022} focus on the differences between the ordinal rankings obtained using these two procedures.

However, a complete pairwise comparison matrix contains $n(n-1)/2$ comparisons, which may be difficult to collect. First, the number of entries is a quadratic function of the number of alternatives. Second, the experts can be unable to compare some items \citep{Harker1987b}. Third, the necessary information might be impossible to acquire, for example, because the pairwise comparisons are derived from the results of matches in sports tournaments and some players or teams have not met each other \citep{BozokiCsatoTemesi2016, ChaoKouLiPeng2018, Csato2013a}.

In this case, the algorithms suggested for complete pairwise comparison matrices can be used only after all missing judgements are estimated. A straightforward approach is considering an optimisation problem where the unknown comparisons are substituted by variables and an inconsistency index provides the objective function \citep{KoczkodajHermanOrlowski1999}.
\citet{ShiraishiObataDaigo1998} and \citet{ShiraishiObata2002} have proposed this idea for the well-established inconsistency index of Saaty. The implied minimisation problem has been analysed and discussed by \citet{BozokiFulopRonyai2010}. \citet{BozokiFulopRonyai2010} also prove the necessary and sufficient condition for the uniqueness of the optimal completion according to the geometric inconsistency index \citep{CrawfordWilliams1985, AguaronMoreno-Jimenez2003}, which minimises a logarithmic least squares objective function. The optimal completion can be obtained by solving a system of linear equations. 

We do not know any result on the relationship of the optimal completions according to these two approaches if the pairwise comparison matrix contains some missing entries. The current paper makes an important contribution to this issue. In particular, it is verified that the two methods lead to the same result if the incomplete pairwise comparison matrix contains at most four alternatives. Our finding is non-trivial because Saaty's inconsistency index $CI$ and the geometric inconsistency index $GCI$ are not functionally dependent for $n=4$ \citep{Cavallo2020}. Unsurprisingly, the theorem does not extend to the case of five alternatives.

The remainder of the study is organised as follows. The theoretical background is presented in Section~\ref{Sec2}. The connection between the two optimal completions is discussed in Section~\ref{Sec3}. Finally, Section~\ref{Sec4} offers concluding remarks.

\section{Basic mathematical definitions} \label{Sec2}

The numerical answers of the decision-maker to questions such as ``How many times alternative $i$ is preferred to alternative $j$?'' are collected into a matrix, but we allow for missing comparisons (indicated by $\ast$), too.
Denote by $\mathbb{R}_+$ the set of positive numbers and by $\mathbb{R}^n_+$ the set of positive vectors of size $n$.

\begin{definition} \label{Def1}
\emph{Incomplete pairwise comparison matrix}:
Matrix $\mathbf{A} = \left[ a_{ij} \right]$ is an \emph{incomplete pairwise comparison matrix} if $a_{ij} \in \mathbb{R}_+ \cup \{ \ast \}$ such that for all $1 \leq i,j \leq n$, $a_{ij} \in \mathbb{R}_+$ implies $a_{ji} = 1 / a_{ij}$ and $a_{ij} = \ast$ implies $a_{ji} = \ast$.
\end{definition}

The set of incomplete pairwise comparison matrices of order $n$ is denoted by $\mathcal{A}_{\ast}^{n \times n}$.

An incomplete pairwise comparison matrix $\mathbf{A} = \left[ a_{ij} \right]$ is called complete if $a_{ij} \neq \ast$ for all $1 \leq i,j \leq n$.


\begin{definition} \label{Def2}
\emph{Weighting method}:
A \emph{weighting method} associates a weight vector $\mathbf{w} \in \mathbb{R}^n_+$ to any incomplete pairwise comparison matrix $\mathbf{A} = \left[ a_{ij} \right] \in \mathcal{A}_{\ast}^{n \times n}$.
\end{definition}

\begin{definition} \label{Def3}
\emph{Logarithmic least squares method} \citep{Kwiesielewicz1996, TakedaYu1995}:
Let $\mathbf{A} = \left[ a_{ij} \right] \in \mathcal{A}_{\ast}^{n \times n}$ be an incomplete pairwise comparison matrix. The weight vector $\mathbf{w} = \left[ w_i \right] \in \mathbb{R}^n_+$ provided by the \emph{logarithmic least squares method} is the optimal solution $\mathbf{w}$ of the following problem:
\begin{align} \label{eq_LLSM}
\min & \sum_{i,j: \, a_{ij} \neq \ast} \left[ \log a_{ij} - \log \left( \frac{w_i}{w_j} \right) \right]^2 \nonumber \\
\text{subject to } & w_i > 0 \text{ for all } i=1,2, \dots n.
\end{align}
\end{definition}

This approach has originally been suggested for complete pairwise comparison matrices \citep{CrawfordWilliams1985, DeGraan1980, deJong1984, Rabinowitz1976, WilliamsCrawford1980}. The objective function \eqref{eq_LLSM} takes only known comparisons into account, that is, the approximation of unknown comparisons is assumed to be perfect.

\begin{definition} \label{Def4}
\emph{Logarithmic least squares optimal completion}:
Let $\mathbf{A} = \left[ a_{ij} \right] \in \mathcal{A}_{\ast}^{n \times n}$ be an incomplete pairwise comparison matrix.
The \emph{logarithmic least squares optimal completion} is $\mathbf{B} = \left[ b_{ij} \right]$ if $b_{ij} = a_{ij}$ for all $a_{ij} \neq \ast$ and $b_{ij} = w_i / w_j$ otherwise, where $\mathbf{w} = \left[ w_i \right]$ is the optimal solution of \eqref{eq_LLSM}.
\end{definition}

Another natural idea is to replace the $m$ missing comparisons with variables $\mathbf{x} \in \mathbb{R}^m_+$, pick up an inconsistency index (see \citet{Brunelli2018} for a comprehensive survey of them), and minimise the inconsistency of the resulting complete pairwise comparison matrix $\mathbf{A}(\mathbf{x})$. It can be seen that logarithmic least squares optimal completion minimises the geometric inconsistency index \citep{CrawfordWilliams1985, AguaronMoreno-Jimenez2003}.

According to \citet{Saaty1977}, the level of inconsistency is a monotonic function of the dominant eigenvalue for any complete pairwise comparison matrix. Thus, the corresponding optimisation problem is as follows.

\begin{definition} \label{Def5}
\emph{Eigenvector method} \citep{ShiraishiObata2002, ShiraishiObataDaigo1998}:
Let $\mathbf{A} = \left[ a_{ij} \right] \in \mathcal{A}_{\ast}^{n \times n}$ be an incomplete pairwise comparison matrix. The weight vector $\mathbf{w} = \left[ w_i \right] \in \mathbb{R}^n_+$ provided by the \emph{eigenvector method} is the optimal solution $\mathbf{w}$ of the following problem:
\begin{align} \label{eq_EM}
\min_{\mathbf{x} \in \mathbb{R}^m_+} \lambda_{\max} \left( \mathbf{A}(\mathbf{x}) \right) \nonumber \\
\lambda_{\max} \left( \mathbf{A}(\mathbf{x}) \right) \mathbf{w} = \mathbf{A}(\mathbf{x}) \mathbf{w}.
\end{align}
\end{definition}

According to Definition~\ref{Def5}, the variables in $\mathbf{x}$ are determined to minimise the dominant eigenvalue of the corresponding complete pairwise comparison matrix, and the priorities are given by the associated right eigenvector as suggested in the AHP methodology.

\begin{definition} \label{Def6}
\emph{Eigenvalue optimal completion}:
Let $\mathbf{A} = \left[ a_{ij} \right] \in \mathcal{A}_{\ast}^{n \times n}$ be an incomplete pairwise comparison matrix.
The \emph{eigenvalue optimal completion} is $\mathbf{B} = \left[ b_{ij} \right]$ if $b_{ij} = a_{ij}$ for all $a_{ij} \neq \ast$ and the value of $b_{ij}$ is determined by the corresponding coordinate of vector $\mathbf{x}$ that is associated with the optimal solution of \eqref{eq_EM} if $a_{ij} = \ast$.
\end{definition}

Graph representation offers a convenient tool to classify incomplete pairwise comparison matrices \citep{SzadoczkiBozokiTekile2022}.

\begin{definition} \label{Def7}
\emph{Graph representation}:
Let $\mathbf{A} = \left[ a_{ij} \right] \in \mathcal{A}_{\ast}^{n \times n}$ be an incomplete pairwise comparison matrix. It is represented by the undirected graph $G = (V,E)$ such that
\begin{itemize}
\item
there is a one-to-one mapping between the vertex set $V = \left\{ 1,2, \dots ,n \right\}$ and the alternatives;
\item
the edge set $E$ is determined by the known comparisons: $(i,j) \in E \iff a_{ij} \neq \ast$.
\end{itemize}
\end{definition}

These concepts can be illustrated by the following example.

\begin{example} \label{Examp1}
Consider the following incomplete pairwise comparison matrix of order four, in which $a_{13}$ (thus $a_{31}$) and $a_{24}$ (thus $a_{42}$) remain undefined:
\[
\mathbf{A} = \left[
\begin{array}{K{3em} K{3em} K{3em} K{3em}}
    1     	& a_{12}  	& \ast   	& a_{14} \\
    a_{21}	& 1       	& a_{23}	& \ast   \\
   \ast		& a_{32}	& 1      	& a_{34} \\
    a_{41}	& \ast  	& a_{43}	& 1 \\
\end{array}
\right].
\]

\begin{figure}[t!]
\centering
\begin{tikzpicture}[scale=1, auto=center, transform shape, >=triangle 45]
\tikzstyle{every node}=[draw,shape=circle];
  \node (n1) at (135:2) {$1$};
  \node (n2) at (45:2)  {$2$};
  \node (n3) at (315:2) {$3$};
  \node (n4) at (225:2) {$4$};

  \foreach \from/\to in {n1/n2,n1/n4,n2/n3,n3/n4}
    \draw (\from) -- (\to);
\end{tikzpicture}

\caption{The graph representation of the incomplete \\ pairwise comparison matrix $\mathbf{A}$ in Example~\ref{Examp1}}
\label{Fig1}
\end{figure}
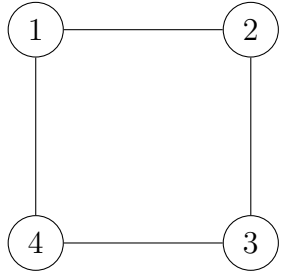

Figure~\ref{Fig1} shows the associated graph $G$.
\end{example}

The necessary and sufficient conditions for the uniqueness of the logarithmic least squares and the eigenvalue optimal completions, respectively, are the same.

\begin{lemma} \label{Lemma1}
The optimal solution to both optimisation problems \eqref{eq_LLSM} and \eqref{eq_EM} is unique if and only if the graph representing the incomplete pairwise comparison matrix is connected.
\end{lemma}

\begin{proof}
See \citet[Theorems~2 and 4]{BozokiFulopRonyai2010}.
\end{proof}

Lemma~\ref{Lemma1} demands a natural requirement for uniqueness since it is impossible to associate priorities for two distinct sets of alternatives if they are not compared, hence, the corresponding graph is disconnected.

Naturally, Lemma~\ref{Lemma1} does not imply that the logarithmic least squares and the eigenvalue optimal completions coincide if graph $G$ is connected. This inspires our research question: When are the corresponding complete pairwise comparison matrices the same?

\section{The main result} \label{Sec3}

The investigation is worth beginning with small problems in the number of alternatives.
The case of $n=3$ is almost trivial. If one comparison is missing and the associated graph is connected, then it should be a spanning tree. Consequently, there exists a unique consistent completion $\mathbf{B} = \left[ b_{ij} \right]$, for which the optimum of \eqref{eq_LLSM} is zero and the optimum of \eqref{eq_EM} equals $n$, that is, both objective functions reach their theoretical minimum. In other words, if $a_{ik} = \ast$, then $b_{ik} = a_{ij} a_{jk}$.

Somewhat surprisingly, the two optimal completions coincide even for $n=4$.

\begin{theorem} \label{Theo1}
Let $\mathbf{A} = \left[ a_{ij} \right] \in \mathcal{A}_{\ast}^{4 \times 4}$ be an incomplete pairwise comparison matrix of size four such that the associated graph $G$ is connected.
The logarithmic least squares and the eigenvalue optimal completions are the same, independently of the number of unknown comparisons.
\end{theorem}

\begin{proof}
First, we show that the pairwise comparison matrix can be considered in the form of 
\begin{equation} \label{4x4xzy}
\mathbf{A} = \left[
\begin{array}{K{2em} K{2em} K{2em} K{2em}}
1             &                1               &             y            &         x \\
1             &                1               &             1            &         z \\
1/y           &                1               &             1            &         1 \\
1/x           &               1/z              &             1            &         1 
\end{array}
\right]
\end{equation}
without loss of generality.

See \citet[Formula~(2)]{FernandesFurtado2022} for the eigenvector method.

The sufficiency of representation in the form \eqref{4x4xzy} for the logarithmic least squares method follows from the fact that if the $i$th row is multiplied by a positive scalar (and, simultaneously, the $i$th column is divided by it), then the corresponding coordinate $w_i$ of the optimal weight vector $\mathbf{w} = \left[ w_i \right]$ is the (same) multiple of the original one before normalisation. 
Multiply the first, second, and fourth rows of a general pairwise comparison matrix 
\[ \label{4x4abcdef}
\left[
\begin{array}{K{2em} K{2em} K{2em} K{2em}}
1             &                a               &             b            &         c \\
1/a           &                1               &             d            &         e \\
1/b           &               1/d              &             1            &         f \\
1/c           &               1/e              &            1/f           &         1 
\end{array}
\right]
\]
by $1 / (ad)$, $1/d$, and $f$, respectively, and divide the first, second, and fourth columns by these numbers to get
\[ \label{4x4abcdef_mod}
\left[
\begin{array}{K{3.5em} K{3.5em} K{3.5em} K{3.5em}}
1              &                1               &      b/ad			       &         c/(adf) 		\\
1              &                1               &             1            &         e/(df)  		\\
ad/b		   &                1               &             1            &                 1      \\
adf/c		   &        df/e		            &             1            &                 1 
\end{array}
\right],
\]
which has exactly the form of \eqref{4x4xzy}.

If the coordinate transformation $x=e^t$, $y=e^u$, $z=e^v$ is applied, $\lambda_{\max}(x,y,z) = \lambda_{\max} \left( e^t,e^u,e^v \right)$ becomes a strictly convex function in $t,u,v \in \mathbb{R}$ \citep[Section~3]{BozokiFulopRonyai2010}. This makes the first-order conditions sufficient for minimality.

Four possible cases shall be discussed.

\vspace{0.25cm}
\noindent
\textbf{Case 1: One comparison ($x$) is missing} \\
If $x$ is missing in \eqref{4x4xzy}, then the logarithmic least squares optimal completion is $x = \sqrt{yz}$, see (\ref{eq_LLSM}), Lemma \ref{Lemma1} and the system of linear equations in the proof of \citet[Theorem~4]{BozokiFulopRonyai2010}.

Based on the calculations of \citet[Formulas~(12) and (13)]{FernandesFurtado2022} for $n=4$, the characteristic polynomial of matrix \eqref{4x4xzy} is $\lambda^4 - 4\lambda^3 + p \lambda + q$, where
\[
p = - z - \frac{1}{z} - y -\frac{1}{y} - \frac{x}{y} - \frac{y}{x} -\frac{x}{z} - \frac{z}{x} + 8, \text{ and}
\]
\[
q = - x - \frac{1}{x} + y + \frac{1}{y} + z + \frac{1}{z} + \frac{x}{y}  + \frac{y}{x}  + \frac{x}{z} + \frac{z}{x} - \frac{y}{z} - \frac{z}{y} - \frac{x}{yz} - \frac{yz}{x} - 2 .
\]

Symbolic calculations by Maple reveal that 
\[
\left. \frac{\partial \lambda_i}{\partial x} \right\vert_{x=\sqrt{yz}} = 0 \qquad \text{ for all } {i=1,2,3,4},
\]
namely, all eigenvalues take an extremal value at $x = \sqrt{yz}$.
Consequently, 
\[
\left.\frac{\partial \lambda_{\max}}{\partial x}\right\vert_{x=\sqrt{yz}} = 0, 
\]
which, taking the argument above into consideration, implies that $\lambda_{\max}$ is indeed minimized at $x = \sqrt{yz}$. \\

\noindent
\textbf{Case 2: Two comparisons ($x, y$) are missing in the same row/column} \\
If $x$ and $y$ are missing in \eqref{4x4xzy}, then the logarithmic least squares optimal completion is given by $x = {z}^{2/3}$ and $y = {z}^{1/3}$.
According to symbolic calculations,
\[
\left.\frac{\partial \lambda_i}{\partial x}\right|_{x={z}^{2/3}, \, y={z}^{1/3}} = 
\left.\frac{\partial \lambda_i}{\partial y}\right|_{x={z}^{2/3}, \, y={z}^{1/3}} 
= 0 \qquad \text{ for all } {i=1,2,3,4},
\]
thus,
\[
\left.\frac{\partial \lambda_{\max}}{\partial x}\right|_{x={z}^{2/3}, \, y={z}^{1/3}} =
\left.\frac{\partial \lambda_{\max}}{\partial y}\right|_{x={z}^{2/3}, \, y={z}^{1/3}} = 
0.
\]

\noindent
\textbf{Case 3: Two comparisons ($y, z$) are missing in different rows/columns} \\
If $y$ and $z$ are missing in \eqref{4x4xzy}, then the logarithmic least squares optimal completion is given by $y = \sqrt{x}$ and $z = {x}^{3/4}$.
According to symbolic calculations,
\[
\left.\frac{\partial \lambda_i}{\partial y}\right|_{y=\sqrt{x}, \, z={x}^{3/4}} = 
\left.\frac{\partial \lambda_i}{\partial z}\right|_{y=\sqrt{x}, \, z={x}^{3/4}}  
= 0 \qquad \text{ for } {i=1,2,3,4},
\]
thus,
\[
\left.\frac{\partial \lambda_{\max}}{\partial y}\right|_{y=\sqrt{x}, \, z={x}^{3/4}} = 
\left.\frac{\partial \lambda_{\max}}{\partial z}\right|_{y=\sqrt{x}, \, z={x}^{3/4}} = 0.
\]

\noindent
\textbf{Case 4: Three comparisons ($x, y, z$) are missing} \\
If $x$, $y$, and $z$ are all missing in \eqref{4x4xzy}, then there is a unique consistent completion given by $x = y = z = 1$, and the minimum of $\lambda_{\max}$ is equal to 4.

The proof is completed because the associated graph $G$ is guaranteed to be disconnected if there are at least four missing comparisons.
\end{proof}

If the graph $G$ representing the incomplete pairwise comparison matrix $\mathbf{A} \in \mathcal{A}_{\ast}^{4 \times 4}$ is disconnected, then both optimisation problems \eqref{eq_LLSM} and \eqref{eq_EM} have an infinite number of solutions.

Theorem~\ref{Theo1} cannot be generalised by increasing the number of alternatives.

\begin{lemma} \label{Lemma2}
The logarithmic least squares and the eigenvalue optimal completions might be different for incomplete pairwise comparison matrices of order five.
\end{lemma}

\begin{proof}
Consider the following pairwise comparison matrix:
\[
\mathbf{A} = \left[
\begin{array}{K{2em} K{2em} K{2em} K{2em} K{2em}}
    1     &  1/2  & 5     &  1/6  & \ast \\
    2     & 1     & 4     &  1/2  &  1/6 \\
     1/5  &  1/4  & 1     &  1/6  &  1/7 \\
    6     & 2     & 6     & 1     &  1/2 \\
    \ast  & 6     & 7     & 2     & 1     \\
\end{array}
\right].
\]
Let $\mathbf{B}$ and $\mathbf{C}$ the logarithmic least squares and the eigenvalue optimal completions, respectively. It can be checked that $b_{15} = 0.1705$ and $c_{15} = 0.1798$, namely, the estimation of the missing comparison between the first and the last alternatives are different according to the two methods. But this is expected as the objective functions to be minimised are different, too.
\end{proof}

\begin{remark} \label{Rem1}
By cloning the second alternative, the example of Lemma~\ref{Lemma2} can be used to verify that the logarithmic least squares and the eigenvalue optimal completions might be different for incomplete pairwise comparison matrices of any order higher than five.
\end{remark}

The incomplete pairwise comparison matrices used as a counterexample in the proof of Lemma~\ref{Lemma2} is minimal with respect to both the number of alternatives and the number of missing entries. However, Lemma~\ref{Lemma2} does not mean that the logarithmic least squares and the eigenvalue optimal completions will always be different if the number of alternatives is at least five. For example, they imply the same completion if the incomplete pairwise comparison matrix can be made consistent with an appropriate choice of the missing entries.

\section{Conclusion} \label{Sec4}

In this paper, we have considered the optimal completion of a pairwise comparison matrix with missing entries if the unknown elements are substituted by variables and the inconsistency of the associated complete matrix is minimised. The logarithmic least squares and the eigenvalue optimal completions are found to be the same if the number of alternatives does not exceed four.

The finding is somewhat surprising because the logarithmic least squares and eigenvector methods can provide different priority vectors for pairwise comparison matrices of order four. Furthermore, some theoretical shortcomings of the eigenvector method such as left-right asymmetry \citep{BozokiRapcsak2008, IshizakaLusti2006, JohnsonBeineWang1979} and Pareto inefficiency \citep{BlanqueroCarrizosaConde2006, BozokiFulop2018} might be a problem if a decision-making problem contains four alternatives. According to Theorem~\ref{Theo1}, this issue becomes relevant only for $n \geq 5$ in the case of incomplete pairwise comparison matrices. Finally, since both approaches lead to the same optimal completion up to $n \leq 4$, one can ``expect'' from other completion methods for pairwise comparison matrices with missing entries to provide the same solution. Consequently, Theorem~\ref{Theo1} may present a kind of axiom for these techniques, eleven of them discussed by \citep{TekileBrunelliFedrizzi2023}.

Our result also brings up several interesting research questions such as:
\begin{itemize}
\item
Are there other classes of incomplete pairwise comparison matrices where the two approaches lead to the same estimation of missing entries?
\item
Does the equivalence hold if the optimal completion is obtained by minimising a third inconsistency index?
\item
When has an incomplete pairwise comparison matrix only one reasonable optimal completion?
\end{itemize}
Hopefully, all these directions will be investigated in the future.


\bibliographystyle{apalike}
\bibliography{All_references}

\end{document}